\documentclass[12pt]{amsart}
\usepackage{amsmath,amssymb,amsthm,graphicx, url, import,verbatim, float,hyperref,ulem, epsfig,epstopdf,caption, xcolor} 
\usepackage{enumerate}
\usepackage{xcolor, relsize}

\newtheorem{lemma}{Lemma}[section]
\newtheorem{proposition}[lemma]{Proposition}
\newtheorem{theorem}[lemma]{Theorem}

\newtheorem{corollary}[lemma]{Corollary}
\newtheorem{question}[lemma]{Question}
\newtheorem{conjecture}[lemma]{Conjecture}

\newtheorem{problem}[lemma]{Problem}
\newcommand{\bcon}{\begin{conjecture}}
\newcommand{\econ}{\end{conjecture}}
\newcommand{\bcor}{\begin{corollary}}
\newcommand{\ecor}{\end{corollary}}
\newcommand{\bdf}{\begin{definition}}
\newcommand{\edf}{\end{definition}}
\newcommand{\benu}{\begin{enumerate}}
\newcommand{\eenu}{\end{enumerate}}
\newcommand{\beq}{\begin{equation}}
\newcommand{\eeq}{\end{equation}}
\newcommand{\bexa}{\begin{example}}
\newcommand{\eexa}{\end{example}}
\newcommand{\bexe}{\begin{exercise}}
\newcommand{\eexe}{\end{exercise}}
\newcommand{\bfac}{\begin{fact}}
\newcommand{\efac}{\end{fact}}
\newcommand{\bite}{\begin{itemize}}
\newcommand{\eite}{\end{itemize}}
\newcommand{\blem}{\begin{lemma}}
\newcommand{\elem}{\end{lemma}}
\newcommand{\bmat}{\begin{matrix}}
\newcommand{\emat}{\end{matrix}}
\newcommand{\bprb}{\begin{problem}}
\newcommand{\eprb}{\end{problem}}
\newcommand{\bpro}{\begin{proposition}}
\newcommand{\epro}{\end{proposition}}

\newcommand{\bque}{\begin{question}}
\newcommand{\eque}{\end{question}}
\newcommand{\brem}{\begin{remark}}
\newcommand{\erem}{\end{remark}}
\newcommand{\bthm}{\begin{theorem}}
\newcommand{\ethm}{\end{theorem}}

\newcommand{\bpr}{\begin{proof}}
\newcommand{\epr}{\end{proof}}

\theoremstyle{definition}
\newtheorem{definition}[lemma]{Definition}
\newtheorem{remark}[lemma]{Remark}
\newtheorem{example}[lemma]{Example}

\newtheorem*{namedtheorem}{\theoremname}
\newcommand{\theoremname}{testing}

\newcommand{\Z}{\mathbb{Z}}

\newcommand{\Q}{\mathbb{Q}}
\newcommand{\C}{\mathbb{C}}

\newcommand{\slC}{\mathrm{SL}_2(\mathbb{C})}

\renewcommand{\S}{\mathcal S}

\title{Torsion in skein modules and shared knot surgeries}

\author{Efstratia Kalfagianni}
\address{Department of Mathematics, Michigan State University, East
Lansing, MI, 48824, USA}
\email{kalfagia@msu.edu}

\def\cX{\mathcal X}

\begin{document}

\begin{abstract}  For every $g>0$, we construct a closed  Haken 3-manifold  $M_g$, whose
 only connected, incompressible surface is a non-separating surface  of genus $g$ and discuss the torsion behavior of the Kauffman bracket skein module $\S(M_g,\Z[A^{\pm 1}])$.
We also use properties of shared surgeries between knots to estimate (and often compute) the dimension
of the Kauffman bracket skein module over  $\Q(A)$, for families of hyperbolic 3-manifolds obtained by 
surgery on double twist knots. We have similar computations
for some surgeries on several knot complements in the SnapPy   census of cusped hyperbolic manifolds
with triangulations up to 9 tetrahedra.
 \end{abstract}

\maketitle

\section{Introduction}
\label{sec:intro}

Let  $M$ denote an oriented, closed $3$-manifold. The   Kauffman bracket skein module $\S(M,\Z[A^{\pm 1}])$,
with coefficients in $\Z([A^{\pm 1}])$, is the quotient of the free  $\Z([A^{\pm 1}])$-module on  isotopy classes of framed unoriented links in $M$, including the empty link $\emptyset$, quotient by the  relations:
\begin{figure}[h]
{\centering
\def \svgwidth{1.00\columnwidth}
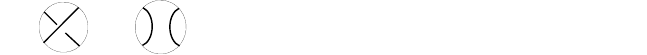}
\end{figure}

It is conjectured
that $\S(M,\Z[A^{\pm 1}])$ is \emph{tame} (e.g. finitely generated)  \cite{DKS2} and free \cite{BD} precisely when $M$ is irreducible and contains no two-sided, incompressible surfaces.
The first examples demonstrating that incompressible surfaces  produce torsion in $\S(M,\Z[A^{\pm 1}])$ where based on 3-manifolds that were not irreducible \cite{HP, Pr} or contained an incompressible torus \cite{V}.
Moreover, the proofs were based on exhibiting explicit torsion elements in the skein module of these 3-manifolds. It was asked whether 2-spheres and essential tori were the only surfaces tin 3-manifolds that are responsible for torsion in the skein module \cite[Question 1.92, (G-E)]{Kirby}.

Let $\zeta$ be a primitive root of unity of order $n$ and let $\Phi_{n}$  denote  the $n$-th cyclotomic polynomial.
In \cite{BD}, Belletti and Detcherry gave  the first examples of irreducible closed 3-manifolds that contain no incompressible tori but have torsion in their skein module. More precisely, \cite[Proposition 3.2 and Corollary 3.1]{BD},
shows that for every $g>1$, there is a closed Haken 3-manifold that contains no incompressible surfaces of genus less that $g$ but $\S(M,\Z[A^{\pm 1}])$ has $\Phi_n(A)$-torsion,
for all odd $n$ and all $n\equiv 2 \,  ({\rm mod}) 4$.
In this note, we construct closed  Haken 3-manifolds each of which contains a unique incompressible surface and their
$\Z[A^{\pm 1}]$-skein module has the same   torsion behavior as in the examples of  \cite{BD}.

To state our result we  recall from \cite{DKS} that a
  $\Z[A^{\pm 1}]$-module $S$ is \underline{tame} if it is a direct sum of cyclic $\Z[A^{\pm 1}]$-modules and $S$ does not contain $\Z[A^{\pm 1}]/(\Phi_{2n}(A))$ as a submodule, 
  for at least one odd $n$. In particular, every finitely generated $\Z[A^{\pm 1}]$-module is tame.

 \begin{theorem} \label{main} For any $g>0$, there exits a closed Haken  3-manifold $M_g$ such that:
\begin{enumerate}[(a)]
\item  For $g=1$, up to isotopy the only connected incompressible surface in $M_1$ is a non-separating torus.

\item For $g>1$, $M_g$ is hyperbolic and  fibers over $S^1$ with fiber of genus $g$, and any connected incompressible surface in $M_g$ is isotopic to the fiber.

\item  The module $\S(M_g,\Z[A^{\pm 1}])$ has $\Phi_{n}(A)$-torsion for every $n$  odd or  $n\equiv 2 \,  ({\rm mod}) 4$.
Hence, in particular, $\S(M_g,\Z[A^{\pm 1}])$  is not tame.
\end{enumerate}
\end{theorem}

At the moment, the mechanism by which surfaces of high genus contribute to torsion in  skein modules is not understood, and it not known if
non-isotopic incompressible surfaces can be distinguished from each other using properties of skein modules. We 
hope that further study of the examples constructed here  will produce explicit torsion elements in $\S(M_g,\Z[A^{\pm 1}])$ and shed some light on these questions.

Given a knot $K\subset S^3$,  let $M_K$ denote the complement of a neighborhood of $K$. Given a slope $r\in \Q\cup \{\infty\}$ we will use $M_K(r)$ to denote the 3-manifold obtained by $r$-surgery along $K$.
We say that two knots $K, K'$ in $S^3$ have \emph{shared surgeries}, if there exist $r_1, r_2\in \Q$ such that $M_K(r_1)$ is homeomorphic to $M_K(r_2)$.
In the second part of the paper, using properties of shared surgeries  of double twist knots, we estimate and often we compute  the dimension $ \dim_{\Q(A)}  \S(M, Q(A))$ for some families of hyperbolic  closed 3-manifolds obtained by Dehn filling along 
families of double twist  knots. For example, we have the following:
\begin{corollary}\label{one} For $|n|>1$,
consider the double twist knots  $L_n:=D(-2, 2n)$.
Then $M_{L_n}(1)$  is hyperbolic and we have  $ \dim_{\Q(A)}  \S(M_{L_n}(1),  Q(A))=4|n|.$
\end{corollary}

We have similar computations for families of surgeries along the twist knots $D(-3, 2n)$ and of some Dehn fillings on several knot complements in the SnapPy   census of  hyperbolic manifolds
with triangulations up to 9 tetrahedra (Tables \ref{tab: qhyp knots2-7} and \ref{tab: qhyp knots8-9}).
The manifolds of Theorem \ref{main} will  also be constructed by surgery on double twist knots and the monodromies of the fibrations  in part (b) are given explicitly (Remark \ref{monodromy}).

\begin{problem}\rm{
Generalize the method of \cite{DW} to compute $\S(M_g, Q(A))$ for the fiber bundles $M_g$ of Theorem \ref{main}(b).}
\end{problem}

\subsection{Acknowledgement}
The research of the author is partially supported by the NSF grants DMS-2603192
 and DMS-2304033.

\medskip

\section{Haken manifolds with unique incompressible surfaces}
Recall that a knot $K\subset S^3$ is called { \emph{small}} if $M_K$ contains  no-closed incompressible surfaces other than  boundary parallel tori.
 We will construct the manifolds $M_g$ in the Theorem \ref{main} by $0$-surgery on a family of small knots $\{K_g\}_{g>0}$.
 Since $K_g$ is small, any closed, incompressible surfaces 
in $M_g:=M_{K_g}(0)$, will come from properly embedded, incompressible, non-boundary parallel surfaces $(S, \partial S)\subset (M_{K_g}, \partial M_{K_g})$, where each component of  $\partial S$ represents  the zero slope on $\partial M_{K_g}$.
In particular, $S$ will be non-separating in $M_{K_g}(0)$.
The families of knots we will use are double twist knots:
\begin{itemize}
\item For $n>0$,  consider the double twist knot $K_g:=D(3, 2n)$ and set $g:=n$ or the twist knot $K_g=D(-3, -2n)$  and set $g:=-n$.
We will take $M_g:=M_{K_g}(0)$, to be the 3-manifold obtained by $0$-surgery on $K_g$.
\begin{figure}[h]
  \includegraphics[scale=0.6]{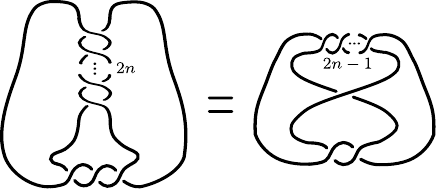} \ \ \ \ \ 
    \includegraphics*[scale=0.6]{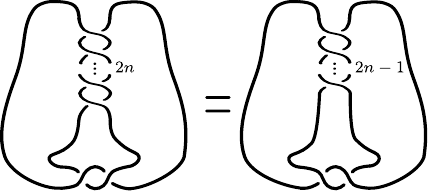}
  \caption{Double twist knots  $D(-3, -2n)$  and $D(-2, -2n)$.}
 \label{fig: alteven}
\end{figure}

\item For $n>1$,  consider the twist knot $K^n_1:=D(2, 2n)$  or the twist knot $K_1^n=D(-2, -2n)$. Again, take  $M_1^n:=M_{K^n_1}(0)$, the 3-manifole
obtained  by $0$-surgery on $K^n_1$.

\item  
Using the conventions of \cite{BC, KM}, for $m,l>0$,
the double twist $D(l, m)$ is the 2-bridge knot $K(\frac{m}{ml - 1})$ 
 associated with
the rational number 
  $\frac{m}{ml - 1} = \frac{1}{l - \frac{1}{m}}$, which corresponds to the  continuous fraction expansion $[l, -m]$.
 Similarly,  $D(l, -m)$  is the 2-bridge knot $K(\frac{m}{ml +1})$ 
 associated with
the rational number 
  $\frac{m}{ml + 1} = \frac{1}{l +\frac{1}{m}}$, which corresponds to the  continuous fraction expansion $[l, m]$.
\end{itemize}

\begin{lemma}\label{manifolds1} For any $g>1$,  $M_g$  hyperbolic and fibers over $S^1$ with fiber of genus $g$. Furthermore, every incompressible, connected, two-sided surface  in $M_g$  is isotopic to the fiber.
\end{lemma}

\begin{proof}Since $ D(-3,- 2n)$  is simply the mirror image of $D(3, 2n)$, we will restrict ourselves in discussing the case $K_g:=D(3, 2g)$, $g>0$.
 Since $K_g$ is a 2-bridge knot,  and hence alternating, but not a $(2, p)$ torus knot, it is hyperbolic \cite{menasco}.
The 
exceptional slopes for all 2-bride knots are known and in particular the only exceptional slope of  $K_g=[3, -2g]$ is $4g$
 \cite[Theorem 1.1 (a)]{BW}. Hence $M_g$ is hyperbolic.

Since 2-bridge knots are  known to be small, any incompressible surfaces in $M_g$ will come from essential surfaces in the complement of $K_g$ that realize the boundary slope zero.
By  \cite{Hath}, essential surfaces in the complement of $K_g$ are determined from continuous fraction expansions  $[a_1, \dots, a_k]$ of the fraction
$\frac{2g}{6g-1}$.  
For twist knots, $D(l, m)$,  $m,l>0$, the
list of continuous fractions with the corresponding boundary slopes they represent, is  worked out in detail in \cite[Table 3]{BC}. For $l=3, m=2g$, they are as follows.
\begin{enumerate}[(i)]
\item $[3, -2g]$ corresponds to boundary slope $4g$.
\item $[2, 2, \ldots, (-1)^{2g} 2]$, where we have a $2g-2$ length sequence of $-2, 2$  corresponds to boundary slope $0$.
\item $[-2, 2, 2g-1]$ corresponds to boundary slope $4g-6$.
\item $[-2, 3, \ldots, (-1)^{2g+1} 2]$ corresponds to boundary slope $-4$.
\end{enumerate}

The  only continuous fraction expansion that corresponds to the zero slope is the one in $(ii)$. 
 Since we have a sequence $[a_1, \dots, a_k]$ with all $|a_i| = 2$, the knot is fibered and the incompressible surface $F_g$ corresponding 
to this sequence is the fiber.
Hence $F_g$ is a minimum genus surface for $K_g$ and the genus of $K_g$ is $g$ \cite{Hath}.
The property   that $|a_i| = 2$, also guarantees that all the incompressible Seifert surfaces of $K_g$ have the same genus and are isotopic to each other \cite[Corollary on Page 230]{Hath}.
This implies that every connected, incompressible surface in $M_g$ is also is isotopic to $F_g(0)$, the image of $F_g$ in $M_{K_g}(0)$. Moreover, the 3-manifold  $M_g$ fibers over $S^1$ with fiber the closed surface $F_g(0)$.
\end{proof}

\begin{lemma}\label{manifolds2} For any $n>0$ the knot $K^n_1:=D(2, 2n)$ has genus one. The 3-manifold $M_1^n:=K^n_1(0)$
is Haken and every connected incompressible surface in it is isotopic to the torus $F^n_1(0)$, that is the image of a genus one Seifert surface of $K^n_1$ in $M_1^n$.
\end{lemma}
\begin{proof}
A genus one surface for $K^n_1$  is obtained by applying Seifert's algorithm to the standard diagram of Figure \ref{fig: alteven}.
Once again,  essential surfaces in the complement of $K_n^1$ are determined from continuous fraction expansions  $[a_1, \dots, a_k]$ of the fraction
$\frac{2n}{4n-1}$ and by  \cite[Table 3]{BC},  we have the following possibilities:

\begin{enumerate}[(i)]
\item $[2, -2n]$ corresponds to boundary slope $0$.
\item $[1, 2, \cdots, (-1)^{2n} 2]$, where we have a $2g-2$ length sequence of $-2, 2$  corresponds to boundary slope $-4n$.
\item $[-2, 2, 1-2m]$ corresponds to boundary slope $-4$.
\item $[-3, -2,  \cdots, (-1)^{2n+1} 2]$ corresponds to boundary slope $-2-4m$.
\end{enumerate}
It follows that the  only continuous fraction expansion that represents the 0-slope is $[2, -2n]$ which represents the genus one Seifert surface.
\end{proof}

\begin{remark} \label{monodromy}On the surface of genus $g$ and with one boundary component $F_{g,1}$, consider the curves $c, a_1, b_1, \cdots, b_{g-1}, a_g$ shown in Figure \ref{fig: generatingcurves} .
 \begin{figure}[h]
    \centering
    \includegraphics[scale=0.5]{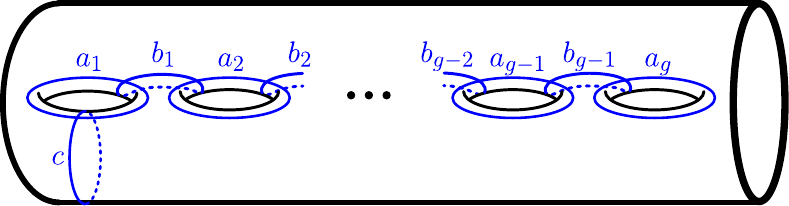}
    \caption{}
    \label{fig: generatingcurves}
  \end{figure}
As pointed out in \cite{KM}, the monodromy of the fibration of the complement of $K_g$ of Lemma \ref{manifolds1}  is the mapping class
$\tau_c\tau_{a_1} \tau^{-1}_{b_1}\tau_{a_2} \cdots \tau^{-1}_{b_{g-1}}\tau_{a_g},$
where for a simple closed curve  $a\subset  F_{g, 1}$, we  let $\tau_a $  denote
the Dehn twist along $a$ and $\tau^{-1}_a$ its inverse and
let ${\tilde \tau}_a$ denote the extension of $\tau_a$ on $F_g:=F_{g,0}$.
The monodromy of $M_g$ is 
$\phi_g={\tilde \tau}_c {\tilde \tau}_{a_1}{\tilde  \tau}^{-1}_{b_1}{\tilde \tau}_{a_2} \cdots {\tilde \tau}^{-1}_{b_{g-1}}{\tilde \tau}_{a_g}.$

\end{remark}

\smallskip

Combining Lemmas \ref{manifolds1} and \ref{manifolds2} we have the following:
\begin{proposition}\label{manifolds} For any $g>0$, there is a closed, Haken  3-manifold $M_g$ such that:
\begin{enumerate}[(a)]
\item  For $g=1$, up to isotopy the only connected incompressible surface in $M_1$ is a non-separating torus.

\item For $g>1$, $M_g$ is hyperbolic and  it fibers over $S^1$ with fiber of genus $g$. Moreover, any connected incompressible surface in $M_g$ is isotopic to the fiber.
\end{enumerate}

\end{proposition}

\medskip

\section{Torsion in skein modules from non-separating surfaces}
For $M_g$ as in statement of Proposition \ref{manifolds},
let
$$\cX(M_g)=\mathrm{Hom}(\pi_1(M_g),\slC)/\hspace*{-.05in}/\slC,$$
denote the $\slC$-character variety of $M_g$, considered as a scheme over $\C$\cite{LM85}.
Let  $\C[\cX(M_g)]$ denote the coordinate ring of $\cX(M_g)$ and
 $X(M_g)$ denote  the algebraic set underlying $\cX(M_g)$.  We use $|X(M)|$ to denote  the cardinality of
$X(M_g)$.

Given a primitive root of unity root of unity of order $n$, say $\zeta$,  consider
$$S_{\zeta}(M_g):=S(M_g, \Z[A^{\pm 1}]) \underset{A=\zeta}{\otimes}   \C.$$


\begin{proposition} \label{torsion} The skein module  $\S(M_g,\Z[A^{\pm 1}])$ has $\Phi_{n}(A)$-torsion, for every $n$ odd  or  $n\equiv 2 \,  ({\rm mod}) 4$.
\end{proposition}

\begin{proof} Let $\zeta$ be a primitive root of unity of order $n$.
Let  $ \S(M_g, \Q(A))$ denote the  Kauffman bracket skein module of $M_g$ defined the same way as before but with coefficients in the field $Q(A)$.
By Theorem 2.1 of \cite{BD} and its proof, if
\begin{equation}\label{criterio}
\dim_{\C}S_{\zeta}(M_g)> \dim_{\Q(A)}  \S(M_g, Q(A)),
\end{equation}
 then $\S(M_g,\Z[A^{\pm 1}])$ has $\Phi_{n}(A)$-torsion.
  Hence,  it is enough to show that inequality (\ref{criterio})  is true for every primitive   root of unity of  of order $n$ or of order  $n\equiv 2 \,  ({\rm mod}) 4$.

First, suppose that $\zeta$ is of order $n$, where  $n\equiv 2 \,  ({\rm mod}) 4$.
By \cite[Theorem 2.1]{DKS}, and its proof, $\dim_{\C}S_{\zeta}(M_g)\geq \dim_{\C}S_{-1}(M_g)$. By \cite{PS},
$S_{-1}(M_g)$ has a natural structure 
 of $\C$-algebra isomorphic with $\C[\cX(M_g)]$. Since $M_g$ contains a non-separating incompressible surface, $b_1(M_g)\geq 1$ and the abelian representations give a curve  $X(M_g)$.  Hence, the dimension $ \dim_{\C}S_{-1}(M_g)$ is infinite.
 It follows that $\dim_{\C}S_{\zeta}(M_g)$ is infinite. On the other hand, $ \dim_{\Q(A)}  \S(M_g, Q(A))$ is finite \cite{GJS}, and hence inequality (\ref{criterio}) holds.
 
 If  $\zeta$ has order $n$, where $n$ is odd,  then  $-\zeta$ has  order $n$, where  $n\equiv 2 \,  ({\rm mod}) 4$. Hence, by above argument, the dimension  $\dim_{\C}S_{-\zeta}(M_g)$
 is infinite. By \cite{Barr}, 
 $\dim_{\C}S_{\zeta}(M_g)=\dim_{\C}S_{-\zeta}(M_g)$. Hence inequality (\ref{criterio}) also holds if $\zeta$ has odd order.
 \end{proof}

 Now combining Propositions \ref{manifolds} and \ref{torsion} gives Theorem \ref{main}.

 \medskip.
 \section{Computations and estimates of  dimensions over $\Q(A)$}
 The following theorem, concerning knots that have shared surgeries with the knot $4_1$,  is a consequence of  \cite[Theorems 1.4,  4.3 and 6.3]{DKS}.
 \begin{theorem}\label{shared}
 Suppose $K$ is  a knot such that there are slopes ${a/b}, p/q\in \Q$ with
 $M_K(a/b)\cong M_{4_1}(p/q)$. Suppose moreover that $p/q\neq 0, \pm 4$. 
 \begin{enumerate} [(a)]
 \item $\S(M_K(a/b),\Z[A^{\pm 1}])$ is finitely generated over $\Z[A^{\pm 1}]$.
 \item We have
 \begin{equation}\label{dim}\dim_{\Q(A)} S({M_K}(a/b), \Q(A))\geq {1\over2} (|4q+p| + |4q-p|)+\epsilon_{p}  +\left\lfloor\frac{|p|}{2}\right\rfloor,
 \end{equation}
 where $\epsilon_{p} =0$ if $p$ is odd,  and $\epsilon_{p} =1$ if $p$ is even.
 Moreover,  (\ref{dim}) is an equality for all but finitely many $p/q$, including all the slopes with $p=1$.
 \end{enumerate}
  \end{theorem}

 The lower bound of  (\ref{dim})  is equal to   $|X( 4_1(p/q)|$  while one always  has $$\dim_{\Q(A)} \S(M, \Q(A))\leq  |\cX(M)|. $$
  It is known that $\cX (4_1(p/q))$ is reduced as scheme over $\C$ for all but finitely many slopes $p/q$ and in particular for all slopes with $p=1$  \cite{DKS}.
  In these
  cases 
 (\ref{dim}) is equality. Using results of \cite{FKBT},  Kitaeff \cite{Kittaeff} studied the skein modules $S({M_K}(r), \Q(A))$ for Dehn fillings of
 2-bridge knots. He showed that for any 2-bridge
 knot $K$, for all but finitely many $r\in \Q$, we have  $\dim_{\Q(A)} \S({M_K}(r))=|\cX(M_K(r))|$. 
 
  In \cite{KM} we studied the problem of which knot complements admit  shared surgeries with the knot $4_1$. Combining results from \cite{KM} with Theorem \ref{shared}, we compute the dimensions of the skein module over $\Q(A)$ for some 
 Dehn fillings along families of twist knots  as well as several 3-manifolds obtained by Dehn fillings of knots in  the SnapPy census. We hope the examples 
 presented here, will be helpful in studies aiming  to understand  the selection 
 principles of slopes to be excluded in Theorem \ref{shared} and in the results of \cite{Kittaeff}.

 \begin{corollary}\label{twist} For $|n|>1$, consider the double twist knots  $J_n:=D(-3, 2n)$ and $L_n:=D(-2, 2n)$ and let  $N_n:=M_{J_n}(1)$ and $N'_n=M_{L_n}(4n+1)$.
\begin{enumerate}[(a)]
   \item The 3-manifolds  $N_n$ and $N'_n$ are hyperbolic.
  \item  For all but at most finitely many $n$, we have $ \dim_{\Q(A)}  \S(N_n, Q(A))=6|n|+1.$
 \item For all $|n|>1$, we have $ \dim_{\Q(A)}  \S(N'_n, Q(A))=4|n|$.
 \end{enumerate}
  \end{corollary}
 \begin{proof}As shown in \cite{KM}  the 3-manifold $M_{4_1}((-4n-1)/n)$ obtained by
  $(-4n-1)/n$-surgery on the figure eight knot,  is homeomorphic to $N_n$. Furthermore, the 3-manifold
 $M_{4_1}(-1/n)$ obtained  $-1/n$-surgery on the figure eight knot,  is homeomorphic  to $N'_n$.
 The exceptional slopes for $4_1$ are $0, \pm1,\pm2, \pm 3$ and $\pm 4$. Since  $|n|>1$, $(-4n-1)/n$ and $-1/n$ are not exceptional and hence $N_n$ and $N'_n$ are hyperbolic.
 Parts (b) and (c) follow by applying Theorem \ref{shared} for $p/q=(-4n-1)/n$ and $p/q=-1/n$, respectively.
 \end{proof}

The list of knots among SnapPy's census of 1267 hyperbolic knot
complements that can be triangulated with fewer than 10 tetrahedra \cite{SnapPy}, and  share surgeries with the figure eight were listed  in \cite[Tables 1 and 2]{KM}. 
Futer, Purcell, and Schleimer originally wrote a software package \cite{FPS}  for testing the cosmetic
surgery conjecture. At our request, they modified their  code to also  allow for testing whether pairs of
cusped $3$-manifolds have common Dehn fillings as well as identifying those fillings \cite[Section 5 and Theorem 5,1]{FPS}.
Running their
code on SnapPy's census they verified the
data given in \cite{KM}. All the shared surgeries found give hyperbolic 3-manifolds.

\begin{table}
   \begin{center}
   \scriptsize 

    \begin{tabular}{ |c|c|c|c|c|}
     \hline
     $K$  & Slope $a/b$& $ \dim_{\Q(A)}  \S(K(a/b), Q(A))$ & Knot \\
     \hline
     $K2_1$    &   Slope $a/b, p/q$             & d(a/b)&  $4_1$ \\ \hline
     $K3_2$      & 5, -5  &$\geq 6$  &$5_2$ \\
                       & 1,  $1/2$        &$=8$  & \\ \hline
     $K4_1$      &  1, $-1/2$       &=8&   $6_1$  \\ \hline
     $K4_2$      &1, $1/3$       &$=12 $ &   $7_2$  \\ \hline
     $K5_2$      &  1, $-1/3$       & =12&  $8_1$ \\ \hline
     $K5_3$      &  1, $1/4$        & $ =16$ &  $9_2$ \\ \hline
     $K5_9$      &  $-2$, $2/3$     &$\geq 14$  &  $10_{132}$ \\ \hline
     $K5_{12}$   & 3, $3/2$        &$\geq 9 $  &  $8_{20}$ \\ \hline
     $K5_{13}$    & 1, $1/3$        & $=12$ & $11n_{38}$  \\ \hline
     $K5_{19}$   & $-7$, $-7/2$    & $\geq  11$  & $6_2$ \\ \hline
     $K5_{20}$   & 9, $-9/2$       &$\geq 12 $   &  $7_3$ \\ \hline
     $K6_1$      & 1, $-1/4$       &$=16$ & $10_1$  \\ \hline
     $K6_2$      & 1, $1/5$        &$=20$ &  $11a_{247}$ \\ \hline
     $K6_8$      & $-3$, $3/5$     & $\geq 21 $  &   \\ \hline
     $K6_9$      & $-2$, $2/5$     &$\geq 22$   &   \\ \hline
     $K6_{23}$  & $-11$, $-11/3$  &$\geq 17$ &  $8_2$ \\ \hline
     $K6_{24}$  & $13$, $-13/3$    &$\geq 18$ & $9_3$  \\ \hline
     $K6_{37}$   & 7, $7/3$        &$\geq 15$   &  $15n_{41127}$ \\ \hline
     $K7_1$      & 1, $-1/5$       & $=20$ &  $12a_{803}$ \\ \hline
     $K7_2$      &  1, $1/6$        & $=24$ & $13a_{3143}$ \\ \hline
     $K7_{10}$   & $-4$, $4/7$     &$\geq 31 $  &   \\ \hline
     $K7_{11}$   &  $-3$, $3/7$     &$\geq 29 $   &   \\ \hline
     $K7_{41}$   &  $-5$, $5/4$     &$\geq 18$  &   \\ \hline
     $K7_{44}$   &  7, $7/5$        &$\geq 23 $   &   \\ \hline
     $K7_{45}$  & $-15$, $-15/4$  &$\geq 23 $   &  $10_2$ \\ \hline
     $K7_{46}$   & $17$, $-17/4$   & $\geq 24$  &  $11a_{364}$ \\ \hline
     $K7_{95}$   &  11, $11/2$      & $\geq 13$ &  $10_{128}$ \\ \hline
     $K7_{96}$   &  13, $13/3$      &$\geq 18 $   &  $11n_{57}$ \\ \hline
     $K7_{98}$    & 14, $14/3$      & $\geq 20$  &  $12n_{243}$ \\ \hline
     $K7_{129}$  &  $-7$, $7/3$     &$\geq 15$  &   \\
      \hline
    \end{tabular}
     \vskip 0.03in
     \caption{Knots in the SnapPy census and skein module dimension bounds for their surgeries.} 
    \label{tab: qhyp knots2-7}
    \end{center}
  \end{table}

Here, we borrow the   tables of shared surgeries from \cite{KM} and use Theorem \ref{shared} to obtain information about the dimension of the $\Q(A)$-Kauffman bracket skein module of these manifolds. 
In the first column of Tables \ref{tab: qhyp knots2-7} and \ref{tab: qhyp knots8-9}
the knot $K$  is given with  the notation of the  SnapPy census.
The second column gives surgery slope pairs $a/b$,
$p/q$,  such that $M_K(a/b)\cong M_{4_1}(p/q)$. Since $4_1$ is amphicheiral, and hence we have $M_{4_1}(-p/q)\cong M_{4_1}(p/q)$,  the
  slopes $p/q$  may be replaced with their negative.
  The last column
identifies some of the knots in the left columns with their notation, say, from \cite{knotinfo}.
 Most of the knots in the first column are not 2-bridge knots
 and have diagrams with high crossing numbers \cite{ChampanerkarKofmanMullen}.

  \begin{table}
       \begin{center}
       \scriptsize 
      \begin{tabular}{ |c|c|c|c|c| }
       \hline
       $K$ & Slope $a/b, p/q$&  $\dim_{\Q(A)}  \S(K(a/b), Q(A))$ & Knot \\
       \hline
       $K8_1$      &  1, $-1/6$      & $=24$ &  $14a_{12741}$ \\ \hline
       $K8_2$      &  1, $1/7$        &$= 28$ &  $15a_{54894}$ \\ \hline
       $K8_9$       & $-5$, $5/9$     &$\geq 38$   &   \\ \hline
       $K8_{10}$    & $-4$, $4/9$    &$\geq 39$   &   \\ \hline
       $K8_{61}$    & 9, $9/7$       &$\geq  32$   &   \\ \hline
       $K8_{62}$   & 11, $11/8$      & $\geq 37$  &   \\ \hline
       $K8_{64}$   & $-19$, $19/5$   & $\geq 29 $  &  $12a_{722}$ \\ \hline
       $K8_{65}$  & $-21$, $21/5$   &$\geq 30 $   &  $13a_{4874}$ \\ \hline
       $K8_{96}$    & 11, $11/5$     & $\geq 25 $  &   \\ \hline
       $K8_{105}$  & $-16$, $16/7$   & $\geq 37$ &   \\ \hline
       $K8_{133}$  & 22, $22/5$      &$\geq  32$  &   \\ \hline
       $K8_{135}$   & 23, $23/5$      &$\geq 31$   & \\ \hline
       $K8_{143}$  & $-13$, $13/4$   &$\geq 23 $&   \\ \hline
       $K8_{145}$   & 1, $1/2$       &$=8 $   &  $14n_{18212}$  \\ \hline
       $K8_{268}$  & 9, $9/4$       &$\geq 20 $  &   \\ \hline
       $K9_1$      & $-1$, $-1/7$   &$\geq 28$   &   \\ \hline
       $K9_2$      & 1, $1/8$       &$\geq 32$  &   \\ \hline
       $K9_8$      & $-6$, $6/11$    &$\geq 48$ &   \\ \hline
       $K9_9$    & 5, $5/11$      &$\geq 46$  &   \\ \hline
       $K9_{83}$   & 13, $13/10$     &$\geq 46 $  &   \\ \hline
       $K9_{85}$   & $-15$, $15/11$  & $\geq 51$ &   \\ \hline
       $K9_{93}$  & 23, $-23/6$    &$\geq 35 $   &  $14a_{12197}$ \\ \hline
       $K9_{94}$   & 25, $-25/6$    &$\geq 36$ &  $15a_{85258}$ \\ \hline
       $K9_{152}$  & 20, $20/9$      &$\geq 47$   &   \\ \hline
       $K9_{155}$  &  $-25$, $25/11$  &$\geq 56 $  &   \\ \hline
       $K9_{242}$   & 31, $31/7$      &$\geq 43$ &   \\ \hline
       $K9_{244}$  & $-32$, $32/7$  &$\geq 45 $  &   \\ \hline
       $K9_{282}$  & $-21$, $21/4$   &$\geq  26$  &   \\ \hline
       $K9_{296}$  & 27, $27/5$      &$\geq 33 $  &   \\ \hline
       $K9_{299}$  & 19, $19/3$      &$\geq 21$ &   \\ \hline
       $K9_{435}$  & $-3$, $3/4$     &$\geq  17$ \\    \hline  
        \end{tabular}
       \vskip 0.03in
      \caption{ Knots in the SnapPy census and skein module dimension bounds for their surgeries.}
       \label{tab: qhyp knots8-9}
      \end{center}
    \end{table}

\begin{remark}
By  Table \ref{tab: qhyp knots2-7},  for $-2$-surgery  the knot $K=10_{132}$, we have $M_{K}(-2)\cong M_{4_1}(2/3)$.
Hence in particular $M_K(-2)$ is  hyperbolic and  non-Haken. By Theorem \ref{shared}, $\S(M, \Z[A^{\pm 1}])$ is a  finitely generated over
$ \Z[A^{\pm 1}]$. It is interesting to note that $a/b=-2$ is a boundary slope for $K$ \cite{knotinfo}
and hence there is a properly embedded essential surface $S$ in $M_K$  with each component  of $\partial S$
having slope $-2$ on $\partial M_K$. However this surface doesn't  survive the Dehn filling along $-2$ slope.
\end{remark}

\bibliographystyle{hamsalpha}
\bibliography{biblio}
\end{document}